\documentclass[11pt]{article}
\usepackage[T1]{fontenc}
\usepackage[utf8]{inputenc}
\usepackage{amsmath,amssymb,amsthm}
\usepackage{geometry}
\usepackage{hyperref}

\newtheorem{theorem}{Theorem}[section]
\newtheorem{proposition}[theorem]{Proposition}
\newtheorem{lemma}[theorem]{Lemma}
\newtheorem{conjecture}[theorem]{Conjecture}

\DeclareMathOperator{\SND}{SND}
\DeclareMathOperator{\WND}{WND}

\begin{document}

\title{\textbf{Hypercube eigenfunctions with two strong nodal domains}}

\author{
Yize He
~and~Qinghong Zhao\thanks{Corresponding author.}
\\[1mm]
\small School of Mathematical Sciences, Huaqiao University, Quanzhou 362021, China}

\date{}
\maketitle

\begin{abstract}

B{\i}y{\i}ko{\u g}lu, Hordijk, Leydold, Pisanski, and Stadler
conjectured that, for every $n\geq3$ and $1\leq i\leq n-2$,
the hypercube $Q_n$ admits a Laplacian eigenfunction with
eigenvalue $2i$ and exactly two strong nodal domains.
In this paper, we confirm this conjecture.

\noindent{\bf Key words}: Hypercube; graph Laplacian; eigenfunction; strong nodal domain\\
\noindent{\bf 2020 Mathematics Subject Classification}: 05C50
\end{abstract}	
\footnotetext{Email-address: hey88040@gmail.com (Y. He) and qzhao@hqu.edu.cn (Q. Zhao).}

\section{Introduction}

All graphs are finite, simple, and undirected. Nodal domains describe how the signs of an eigenfunction are arranged on a graph. Beyond bounding their number in terms of spectral data, one can ask how few nodal domains occur within a prescribed eigenspace. For a connected graph, every Laplacian eigenfunction with a nonzero eigenvalue has both positive and negative values, so two is the smallest possible number of strong nodal domains. Determining whether this minimum is attained requires both sign supports to be connected. We study this question for the hypercube, where the spectrum is explicit but the connectivity of the sign supports depends on the choice of eigenfunction.
For a real-valued function $f$ on $V(G)$, write
\[
S_+(f)=\{x\in V(G):f(x)>0\},\qquad
S_-(f)=\{x\in V(G):f(x)<0\}.
\]
The positive and negative strong nodal domains of $f$ are the connected components of the induced subgraphs $G[S_+(f)]$ and $G[S_-(f)]$, respectively. Weak nodal domains are defined using the sets $\{f\geq0\}$ and $\{f\leq0\}$, retaining only components that contain a vertex at which $f$ is nonzero. We denote the total numbers of strong and weak nodal domains by $\SND(f)$ and $\WND(f)$.

Discrete nodal domain theory relates sign structure to eigenvalue order and multiplicity. Early results include \cite{Friedman1993,DuvalReiner1999}; upper bounds for generalized Laplacians are established in \cite{DGLS2001}, with sharper conclusions for trees in \cite{Biyikoglu2003} and lower bounds in \cite{Berkolaiko2008,XuYau2012}. The monograph \cite{BLS2007} provides a systematic account. More recent work treats signed graphs and symmetric matrices \cite{GeLiu2023,McKenzieUrschel2024} and applies nodal domain methods to equiangular lines \cite{GeLiu2025}. 

For a set $S$ of size $q\geq2$, the Hamming graph $H(n,q)$ has vertex set $S^n$, with two vertices adjacent when they differ in exactly one coordinate. Its Laplacian eigenvalues are $qi$, $0\leq i\leq n$. The hypercube is $Q_n=H(n,2)$, which we represent on $\{-1,1\}^n$; its eigenvalues are therefore $2i$, $0\leq i\leq n$.

B{\i}y{\i}ko{\u g}lu, Hordijk, Leydold, Pisanski, and Stadler \cite{Biyikoglu2004} study the minimum numbers of weak and strong nodal domains in each hypercube eigenspace. They show that, for every $1\leq i\leq n-1$, there exists a $2i$-eigenfunction with exactly two weak nodal domains. For strong nodal domains, they establish the corresponding existence result for $1\leq i\leq n/2$.

The two largest eigenvalues impose restrictions. For $n\geq3$, every eigenfunction with eigenvalue $2(n-1)$ has at least $n$ strong nodal domains \cite{Biyikoglu2004}. 
Every Laplacian eigenfunction of $Q_n$ with eigenvalue $2n$
has exactly $2^n$ strong nodal domains
\cite[Theorem~2]{Biyikoglu2004}.
Thus, neither $i=n-1$ nor $i=n$ admits an eigenfunction
with exactly two strong nodal domains when $n\geq3$.

Motivated by numerical experiments, B{\i}y{\i}ko{\u g}lu et al.\ propose the following conjecture.

\begin{conjecture}
\label{conj:hypercube}
{\normalfont\cite{Biyikoglu2004}}
For every $n\geq3$ and $1\leq i\leq n-2$, there exists a Laplacian eigenfunction $f$ on $Q_n$ with eigenvalue $2i$ such that $\SND(f)=2$.
\end{conjecture}

Valyuzhenich and Vorob\'ev \cite{Valyuzhenich2025}  prove the conjecture for $i\leq \frac{2}{3}(n-\frac12)$
when $i$ is odd and for $i\leq \frac{2}{3}(n-1)$ when $i$ is even.
Their construction uses equitable $2$-partitions. Constructions and parameter restrictions for these partitions on hypercubes and Hamming graphs are studied in \cite{FdF2007Colorings,BKMTV2021}. The range attained by their method is tied to the correlation-immunity bound of Fon-Der-Flaass \cite{FdF2007Immunity}. Valyuzhenich and Vorob\'ev also obtain eigenfunctions with exactly two strong nodal domains on $H(n,3)$ for every $n\geq2$ and $1\leq i\leq n-1$, and on $H(n,q)$ for every $n\geq1$, $q\geq4$, and $1\leq i\leq n$, with Laplacian eigenvalue $qi$ in each case. In this paper, we prove Conjecture~\ref{conj:hypercube}.

\begin{theorem}\label{thm:main}

For all integers $n\ge3$ and $1\le i\le n-2$, there exists a function $f:\{-1,1\}^n\to\mathbb{R}$ satisfying $L_{Q_n}f=2if$ such that $S_+(f)$ and $S_-(f)$ are both nonempty and induce connected subgraphs of $Q_n$. In particular, $\SND(f)=2$.
\end{theorem}

\section{Construction and proof}\label{sec:proof}

Let \[
 E_i(n)=\{f\colon V(Q_n)\to\mathbb R:L_{Q_n}f=2if\}
\] be the eigenspace of the Laplacian operator on the hypercube $Q_n$corresponding to the eigenvalue $2i$. When $i = n-2$, define the polynomial representing the weighted sum of adjacent terms
\[
 q_n(x)=\sum_{j=1}^{n-1}a_jx_jx_{j+1}.
\]
The coefficients $a_j$ are specified below. Set
\[
p_n(x)=\prod_{j=1}^n x_j,
\qquad f_n(x)=p_n(x)q_n(x).
\]
Since $x_j^2=1$, we have
\begin{align*}
 f_n(x)
 &=\left(\prod_{k=1}^n x_k\right)q_n(x)\\
 &=\sum_{j=1}^{n-1}a_jx_jx_{j+1}\prod_{k=1}^n x_k
 =\sum_{j=1}^{n-1}a_j
   \prod_{k\in[n]\setminus\{j,j+1\}}x_k.
\end{align*}

For any subset $S \subseteq [n] = \{1, 2, \dots, n\}$, define the character monomial $\chi_S(x) = \prod_{j \in S} x_j$. A neighbor $y \sim x$ of a vertex $x$ is obtained by flipping a single coordinate of $x$: if the flipped coordinate belongs to $S$, then $\chi_S(y) = -\chi_S(x)$, in which case $\chi_S(x) - \chi_S(y) = 2\chi_S(x)$; if the flipped coordinate does not belong to $S$, then $\chi_S(y) = \chi_S(x)$, and the difference is 0. Since there are $|S|$ coordinates in the set $S$, there are exactly $|S|$ neighbors corresponding to flipping a coordinate in $S$; therefore,
\[
 L_{Q_n}\chi_S=2|S|\chi_S.
\]
This indicates that each $\chi_S$ is an eigenvector of the Laplacian operator, corresponding to the eigenvalue $2|S|$. 
Each monomial in $f_n$ has degree $n-2$.
Hence $f_n\in E_{n-2}(n)$.
By the following extension lemma, it suffices to construct,
for every $d\ge3$, a function $f_d\in E_{d-2}(d)$
whose positive and negative supports are nonempty and connected.
\begin{lemma}
\label{lem:extension}
{\normalfont\cite{Valyuzhenich2025}}
Let $3\le d\le n$, and let $f_d\in E_{d-2}(d)$
have nonempty connected positive and negative supports.
Then the function
$
F(x_1,\dots,x_n)=f_d(x_1,\dots,x_d)
$
belongs to $E_{d-2}(n)$ and has nonempty connected
positive and negative supports.
\end{lemma}

\begin{proof}
Fix the target dimension $n \ge 3$. We have $L_{Q_d} f_d = 2(d-2)f_d$. If $d = n$, the conclusion holds trivially. Assume that $d < n$.
Write $x = (u, v)$, where $u \in V(Q_d)$ and $v \in V(Q_{n-d})$. Define $F(u, v) := f_d(u)$. Expanding the graph Laplacian yields
\begin{align*}
(L_{Q_n}F)(u,v)
&= \sum_{X \sim_{Q_n} (u,v)} \bigl(F(u,v) - F(X)\bigr) \\
&= \sum_{u' \sim_{Q_d} u} \bigl(F(u,v) - F(u',v)\bigr) + \sum_{v' \sim_{Q_{n-d}} v} \bigl(F(u,v) - F(u,v')\bigr) \\
&= (L_{Q_d} f_d)(u) + 0 = 2(d-2) F(u, v).
\end{align*}
Thus we have $F \in E_{d-2}(n)$. 

Regarding connectivity, the positive support set is $\{F > 0\} = \{f_d > 0\} \times V(Q_{n-d})$. Consider any two points $(u, v), (u', v') \in \{F > 0\}$:
Since $u,u'\in S_+(f_d)$ and $Q_d[S_+(f_d)]$ is connected, there exists a path $u=u_0\sim u_1\sim\cdots\sim u_s=u'$
such that $f_d(u_j)>0$ for every $0\le j\le s$. It follows that
\[
(u,v) = (u_0, v) \sim (u_1, v) \sim \cdots \sim (u_s, v) = (u', v).
\]
Next, using the connectivity of the hypercube $Q_{n-d}$ itself, we obtain a path connecting $v$ and $v'$—namely $v = v_0 \sim v_1 \sim \cdots \sim v_t = v'$ by successively flipping the coordinates where $v$ and $v'$ differ. This gives the path
\[
(u', v) = (u', v_0) \sim (u', v_1) \sim \cdots \sim (u', v_t) = (u', v').
\]
Since $F(u', v_k) = f_d(u') > 0$ for any $0 \le k \le t$, this path lies entirely within the positive support.
Concatenating these two path segments yields a path from $(u, v)$ to $(u', v')$ that lies entirely within the region $\{F > 0\}$. Therefore, the positive support induces a connected subgraph in $Q_n$. The argument for the negative support is identical.
\end{proof}

By Lemma ~\ref{lem:extension}, to prove Theorem~\ref{thm:main}, it suffices to construct, for each integer $d \ge 3$, a function $f_d \in E_{d-2}(d)$ whose positive and negative supports are both nonempty and connected.
We now specify  $q_n(x)$. Let $m = n-1$ and $l = \lfloor \frac{m}{2} \rfloor$. Define the adjacent product variables as
\[
y_j = x_j x_{j+1} \in \{-1, 1\}, \quad j = 1, 2, \dots, n-1.
\]
Group the adjacent variables into pairs:
$
(y_1, y_2), \ (y_3, y_4), \ \dots, \ (y_{2l-1}, y_{2l}).
$
Each pair sum belongs to $\{-2,0,2\}$.
Define the weighted function $q_n(x)$ based on the parity of the dimension $n$:
When $n$ is even: $m = n-1 = 2l+1$,
\[
q_n(x) = \sum_{r=1}^l 2^{l-r} (y_{2r-1} + y_{2r}) + y_m.
\]
When $n$ is odd: $m = n-1 = 2l$,
\[
q_n(x) = \sum_{r=1}^l 2^{l-r} (y_{2r-1} + y_{2r}) + \sum_{r=1}^l 2^{r-l-1} y_{2r}.
\]

For $q_n(x)$, under this definition, if all pair sums $y_{2r-1}+y_{2r}$ vanish, a single correction term $y_m$ exists in the even case, while a correction term $\sum_{r=1}^l 2^{r-l-1} y_{2r}$ exists in the odd case. So $q_n(x) \neq 0$ holds everywhere.

If the first pair encountered with identical components is the $r$-th pair, and both its terms are $s \in \{-1, 1\}$, then the principal contributions from all preceding pairs are zero, and the contribution of this pair to $q_n(x)$ is $s \cdot 2^{l-r+1}$. The contribution from the later pairs has absolute value at most
$
\sum_{k=r+1}^l 2 \cdot 2^{l-k} = 2^{l-r+1} - 2.
$
In this case, if $n$ is even, the absolute value of the remaining contribution is at most
$
2^{l-r+1} - 2 + 1 = 2^{l-r+1} - 1 < 2^{l-r+1};
$
if $n$ is odd, the absolute value of the remaining contribution is at most
$
2^{l-r+1}-2+\sum_{k=1}^l  \left|2^{k-l-1} y_{2k} \right|= 2^{l-r+1}-2+1 - 2^{-l} <2^{l-r+1}- 1,
$
so the total absolute value of the contributions from the correction term and subsequent terms is also strictly less than $2^{l-r+1}$. Thus, the first pair with identical components completely determines the sign of $q_n(x)$. 

If $y_{2r-1} + y_{2r} = 0$ holds for all $1 \le r \le l$:
When $n$ is even, the sum of the first $l$ pairs is $0$, leaving only the correction term $y_m$; clearly, $\operatorname{sgn} q_n(x) = y_m$.
When $n$ is odd, only the correction term
$\sum_{k=1}^l 2^{k-l-1} y_{2k}$ remains.
Its last term is $\frac{1}{2} y_{2l} = \frac{1}{2} y_m$,
while the sum of the absolute values of the other terms is at most
$\frac{1}{2} - 2^{-l} < \frac{1}{2}$;
thus, the sign is entirely determined by the last term $y_m$,
i.e., $\operatorname{sgn} q_n(x) = y_m$.
Formally, this property can be characterized by the following recurrence.
\begin{lemma}\label{lem:sgn}
Let $n \ge 5$ and $x = (x_1, x_2, z) \in \{-1, 1\}^n$, where $z = (x_3, \dots, x_n)$. Then $$\operatorname{sgn} q_n(x_1, x_2, z) = \begin{cases} x_1 x_2, & x_1 = x_3, \\ \operatorname{sgn} q_{n-2}(z), & x_1 = -x_3. \end{cases}$$

\end{lemma}

\begin{proof}
If $x_1=x_3$, then $y_1=y_2=x_1x_2$,
so the first pair determines the sign of $q_n(x)$. If $x_1=-x_3$, then $y_1+y_2=0$.
The remaining pairs, in their original order, are exactly
those associated with $z$, and the final adjacent product
is unchanged.
The sign rule established above therefore gives
$
\operatorname{sgn}q_n(x)
=\operatorname{sgn}q_{n-2}(z).
$
\end{proof}

For the next proof, set
\[
p_n(x) = \prod_{j=1}^n x_j, \quad f_n(x) = p_n(x) q_n(x), \quad g_n(x) = \operatorname{sgn} f_n(x).
\]
Since $q_n(x) \neq 0$ everywhere, $g_n(x) \in \{-1, 1\}$ is well-defined everywhere.
Fix \(n\ge5\) and put \(d=n-2\). For $Z = (z_1, \dots, z_d) \in Q_d$ and $X = (x_1, x_2, Z) \in Q_n$, we partition the vertices into two classes based on whether $x_1 = z_1$ or $x_1 = -z_1$:
\[
A_b(Z) = (z_1, b, Z), \quad B_b(Z) = (-z_1, b, Z), \quad b \in \{-1, 1\}.
\]
Let $\pi(Z) = \prod_{j=1}^d z_j$, $h(Z) = \operatorname{sgn} q_d(Z)$, and $\gamma(Z) = \operatorname{sgn} f_d(Z) = \pi(Z) h(Z)$.  Lemma~\ref{lem:sgn} gives:
$
\operatorname{sgn} q_n(A_b(Z)) = z_1 b , \quad \operatorname{sgn} q_n(B_b(Z)) = h(Z).
$
Thus,
$\operatorname{sgn} f_n(A_b(Z)) = (z_1 \cdot b \cdot \pi(Z)) \cdot (z_1 b) = \pi(Z), $
$\operatorname{sgn} f_n(B_b(Z)) = (-z_1 \cdot b \cdot \pi(Z)) \cdot h(Z) = -z_1 b \gamma(Z).$

\begin{proposition}\label{prop:connectedness}
For $n\ge3$ and $\tau\in\{-1,1\}$, let
$
V_\tau(n)=\{x\in\{-1,1\}^n:g_n(x)=\tau\}.
$
Then $V_\tau(n)$ is nonempty and $Q_n[V_\tau(n)]$
is connected.
\end{proposition}

\begin{proof}
We argue by induction on $n$, with base cases $n=3,4$.

Consider the case $n=3$. Here, $m = 2$ and $l = 1$. By definition, we have:
$q_3(x) = y_1 + \frac{3}{2}y_2,$
$f_3(x) = x_3 + \frac{3}{2} x_1.$
Thus, the sign of $f_3(x)$ is determined entirely by $x_1$. After fixing $x_1 \in \{-1, 1\}$, the four vertices formed by the remaining two coordinates are:
$
(x_1, 1, 1), \quad (x_1, -1, 1), \quad (x_1, -1, -1), \quad (x_1, 1, -1).
$
These four vertices form a cycle \[(x_1, 1, 1) \sim (x_1, -1, 1) \sim (x_1, -1, -1) \sim (x_1, 1, -1) \sim (x_1, 1, 1).\] Thus both the positive and negative supports induce
connected subgraphs.

When $n=4$, we have $m = 3$ and $l = 1$:
$
q_4(x) = y_1 + y_2 + y_3, \quad f_4(x) = x_3 x_4 + x_1 x_4 + x_1 x_2.
$
Direct verification shows that:
$V_+(4) = \{f_4 > 0\}$ contains $8$ vertices, forming the following cycle:
\[
(+,+,+,+)\sim(+,+,-,+)\sim(+,+,-,-)\sim(-,+,-,-)\sim(-,-,-,-)\sim
\]
\[
(-,-,+,-)\sim(-,-,+,+)\sim(+,-,+,+)\sim(+,+,+,+);
\]
$V_-(4) = \{f_4 < 0\}$ also contains $8$ vertices, forming the following cycle:
\[
(+,+,+,-)\sim(+,-,+,-)\sim(+,-,-,-)\sim(+,-,-,+)\sim(-,-,-,+)\sim
\]
\[(-,+,-,+)\sim(-,+,+,+)\sim(-,+,+,-)\sim(+,+,+,-).
\]

Therefore, the conclusion holds for both $n=3$ and $n=4$. Let $d = n-2 \ge 3$. By the inductive hypothesis, the set of vertices $V_\tau(d)$ with the same sign in $Q_d$ induces a connected subgraph; we now prove that the set of vertices $V_\tau(n)$ with the same sign in $Q_n$ also induces a connected subgraph. We first show that all type-B vertices lie in a single connected component of \(Q_n[V_\tau(n)]\), and then connect every type-A vertex in $V_\tau(n)$ to this component.
For a vertex $A_b(Z)$, its sign  is always $\pi(Z)$. For a given target sign $\tau \in \{-1, 1\}$, a type-B vertex satisfies $\operatorname{sgn} f_n(B_b(Z)) = \tau$ if and only if $-z_1 b \gamma(Z) = \tau$. Let
$
b_\tau(Z) := -\tau z_1 \gamma(Z), \quad \widehat{B}_\tau(Z) := B_{b_\tau(Z)}(Z) \in V_\tau(n).
$
Thus, for each $Z \in Q_d$, there exists a unique $b \in \{-1, 1\}$ such that the vertex $\widehat{B}_\tau(Z)$ has sign $\tau$.
Fix $\tau\in\{-1,1\}$.
We first connect type-B vertices with the same value of
$\gamma$, and then connect the two resulting groups.
All paths are taken in $Q_n[V_\tau(n)]$.
 Let $Z, Z' \in Q_d$ have $\gamma(Z) = \gamma(Z')$. By the induction hypothesis, it suffices to consider adjacent vertices \(Z,Z'\) with \(\gamma(Z)=\gamma(Z')\).
 
If the flipped coordinate is at position $j \ge 2$, then $z_1 = z_1'$ and $b_\tau(Z) = b_\tau(Z')$; thus, $\widehat{B}_\tau(Z)$ and $\widehat{B}_\tau(Z')$ differ by only one coordinate and are adjacent in $Q_n$.
If the flipped coordinate is the first one, $z_1$ (i.e., $z_1' = -z_1$), then $\pi(Z') = -\pi(Z)$. Since their $\pi$ values have opposite signs, one of them must equal $\tau$. Assume $\pi(Z) = \tau$; then, according to the sign properties of type-A vertices, $\operatorname{sgn} f_n(A_b(Z)) = \pi(Z) = \tau$ holds for both $b \in \{-1, 1\}$. Note that:
$\widehat{B}_\tau(Z) = (-z_1, b_\tau(Z), Z) \sim A_{b_\tau(Z)}(Z) = (z_1, b_\tau(Z), Z) \sim A_{-b_\tau(Z)}(Z) = (z_1, -b_\tau(Z), Z).$
By definition, $-b_\tau(Z) = \tau z_1 \gamma(Z) = -\tau z_1' \gamma(Z') = b_\tau(Z')$, and $-z_1' = z_1$. Therefore, $A_{-b_\tau(Z)}(Z)$ and $\widehat{B}_\tau(Z') = (-z_1', b_\tau(Z'), Z')$ differ only in the sign of the third coordinate, making them directly adjacent. Thus, there exists a path lying entirely within $V_\tau(n)$:
\[
\widehat{B}_\tau(Z) \sim A_{b_\tau(Z)}(Z) \sim A_{-b_\tau(Z)}(Z) \sim \widehat{B}_\tau(Z').
\]
Therefore, type-B vertices corresponding to the same value of $\gamma(Z)$ are connected to each other within $V_\tau(n)$.

Connection between type-B vertices with different $\gamma(Z)$ values. Let $u=(1,\dots,1)$ and $v=(-1,1,\dots,1)$.
The definition of $q_d$ gives $h(u)=h(v)=1$.
Since $\pi(u)=1$ and $\pi(v)=-1$, we obtain
$
\gamma(u)=1,\gamma(v)=-1,
 b_\tau(u)=b_\tau(v)=-\tau.
$
Let $w = (1, \dots, 1) \in \{-1, 1\}^{d-1}$; then $\widehat{B}_\tau(u) = (-1, -\tau, 1, w)$ and $\widehat{B}_\tau(v) = (1, -\tau, -1, w).$

If $\tau = 1$, $\pi(u) = 1 = \tau$, and the point $A_{-1}(u) = (1, -1, 1, w) \in V_1(n)$ forms a path:
\[
\widehat{B}_1(u) \sim A_{-1}(u) \sim \widehat{B}_1(v);
\]

if $\tau = -1$, $\pi(v) = -1 = \tau$, and the point $A_1(v) = (-1, 1, -1, w) \in V_{-1}(n)$ forms a path:
\[
\widehat{B}_{-1}(u) \sim A_1(v) \sim \widehat{B}_{-1}(v).
\]
Since type-B vertices with the same $\gamma$ value are connected within $V_\tau(n)$, the connection between $\widehat{B}_{\tau}(u)$ and $\widehat{B}_{\tau}(v)$ implies that the sets of type-B vertices with different $\gamma$ values are mutually connected within $V_\tau(n)$. Hence all type-B vertices lie in the same connected component of \(Q_n[V_\tau(n)]\).

Finally, consider type-A points: any type-A point satisfying $\operatorname{sgn} f_n(A_b(Z)) = \tau$ must satisfy $\pi(Z) = \tau$.  $A_{b_\tau(Z)}(Z)$ and $\widehat{B}_\tau(Z)$ differ only in their first coordinate and are directly adjacent; meanwhile, the other  type-A point $A_{-b_\tau(Z)}(Z)$ differs from $A_{b_\tau(Z)}(Z)$ only in its second coordinate and is likewise directly adjacent:
\[
A_{-b_\tau(Z)}(Z) \sim A_{b_\tau(Z)}(Z) \sim \widehat{B}_\tau(Z).
\]
Hence every type-A vertex in $V_\tau(n)$ belongs to
the component containing all type-B vertices.
Therefore $Q_n[V_\tau(n)]$ is connected.
\end{proof}

\begin{proof}[Proof of Theorem~\ref{thm:main}]
For each $d\ge3$, the construction gives $f_d\in E_{d-2}(d)$.
Proposition~\ref{prop:connectedness} shows that its positive and
negative supports are nonempty and connected. Given $n\ge3$ and
$1\le i\le n-2$, set $d=i+2$ and apply Lemma~\ref{lem:extension}.
The function
\[
 F(x_1,\dots,x_n)=f_d(x_1,\dots,x_d)
\]
belongs to $E_i(n)$ and has exactly two strong nodal domains.
\end{proof}

\end{document}